\documentclass{article}
\usepackage{amsmath,amssymb,amsfonts,amsthm,graphicx,a4wide}

\newcommand{\e}{\varepsilon}
\newcommand{\A}{\mathcal{A}}
\newcommand{\B}{\mathcal{B}}
\newcommand{\U}{\mathcal{U}}

\newcommand{\N}{\mathbb{N}}

\renewcommand{\le}{\leqslant}
\renewcommand{\leq}{\leqslant}
\renewcommand{\ge}{\geqslant}
\renewcommand{\geq}{\geqslant}
\renewcommand{\rho}{\varrho}

\DeclareMathOperator{\Card}{Card}
\DeclareMathOperator{\Fact}{Fact}

\newtheorem{theorem}{Theorem}

\newtheorem{corollary}[theorem]{Corollary}
\newtheorem{definition}[theorem]{Definition}

\renewcommand{\i}{\item[\rm{(i)}]}
\newcommand{\ii}{\item[\rm{(ii)}]}
\newcommand{\iii}{\item[\rm{(iii)}]}

\begin{document}

\begin{center}
\begin{Large}
\begin{bf}
On the Entropy and Letter Frequencies\\[0.75ex] of Powerfree Words\vspace{3ex}
\end{bf}
\end{Large}

{\large Uwe Grimm and Manuela Heuer\vspace{1.5ex}}

\begin{it}
Department of Mathematics and Statistics,
The Open University,\\ 
Walton Hall, Milton Keynes \mbox{MK7 6AA}, United Kingdom\vspace{1.5ex}
\end{it}
\end{center}

\begin{quote}
We review the recent progress in the investigation of
powerfree words, with particular emphasis on binary cubefree and
ternary squarefree words. Besides various bounds on the entropy, we
provide bounds on letter frequencies and consider their empirical
distribution obtained by an enumeration of binary cubefree words up to
length $80$.
\end{quote}\vspace{1ex}

\section{Introduction}

The interest in combinatorics on words goes back to the work of Axel
Thue at the beginning of the 20th century \cite{Thue:1977}. He showed, in
particular, that the famous morphism
\begin{equation}\label{eq:tm}
    \rho\!:\quad \begin{array}{ccc} 0 &\mapsto& 01 \\
                                 1 &\mapsto& 10 \end{array},
\end{equation}
called Thue-Morse morphism since the work of Morse \cite{Morse:1921},
is cubefree. Its iteration on the initial word $0$ produces an
infinite cubefree word
\[
    0110100110010110100101100110100110010110011010010110100110010110 \ldots
\]
over a binary alphabet, which means that it does not contain any
subword of the form $0^3 = 000$, $1^3 = 111$, $(01)^3 = 010101$,
$(10)^3=101010$ and so forth. Moreover, the statement that the
morphism is cubefree means that it maps any cubefree word to a
cubefree word, so it preserves this property.  Generally, the
iteration of a powerfree morphism is a convenient way to produce
infinite powerfree words.

The investigation of powerfree, or more generally of pattern-avoiding
words, is one particular aspect of combinatorics on words; we refer
the reader to the book series
\cite{Lothaire:1997,Lothaire:2002,Lothaire:2005} for a comprehensive
overview of the area, including algebraic formulations and
applications. The area has attracted considerable activity in the past
decades
\cite{Zech:1958,Pleasants:1970,Bean:1979,Crochemore:1982,Shelton:1983,
Brandenburg:1983,Brinkhuis:1983,Leconte:1985a,Leconte:1985b,
Seebold:1985,Kobayashi:1986,Keraenen:1987,
Baker:1989,Currie:1993,Kolpakov:1999,Grimm:2001,Currie:2002,Richomme:2002,
Kucherov:2003}, and continues to do so, see
\cite{Karhu:2004,Richard:2004,Ochem:2006,Richomme:2007,Kolpakov:2007,
Ochem:2007a,Chalopin:2007,Ochem:2007b,Khalyavin:2007} for some recent
work. Beyond the realm of combinatorics on words and coding theory,
substitution sequences, such as the Thue-Morse sequence, have been
investigated for instance in the context of symbolic dynamics
\cite{Queffelec:1987,Fogg:2002,Allouche:2003} and aperiodic order
\cite{Moody:1997}, to name but two. In the latter case, one is
interested in systems which display order without periodicity, and
substitution sequences often provide paradigmatic models, which are
used in many applications in physics and materials science. However,
sequences produced by a substitution such as in Eq.~\eqref{eq:tm} have
subexponential complexity and hence zero combinatorial entropy, cf.\
Definition~\ref{def:entropy} below. A natural generalisation to
interesting sets of positive entropy is provided by powerfree or
pattern-avoiding words.

In this article, we review the recent progress on powerfree words,
with emphasis on the two `classic' cases of binary cubefree and
ternary squarefree words. We include a summary of relevant results
which are scattered over 25 years of literature, and also discuss some
new results as well as conjectures on cubefree morphisms and letter
frequencies in binary cubefree words.

The first term of interest is the combinatorial entropy of the set of
powerfree words. Due to the fact that every subword of a powerfree
word is again powerfree, the entropy of powerfree words exists as a
limit. It is a measure for the exponential growth rate of the number
of powerfree words of length $n$. Unfortunately, neither an explicit
expression for the entropy of $k$-powerfree words nor an easy way to
compute it numerically is known. Nevertheless, there are several
strategies to derive upper and lower bounds for this limit. Upper
bounds can be obtained, for example, by enumeration of all powerfree
words up to a certain length, or by the derivation of generating
functions for the number of powerfree words, see Section~4. Until
recently all methods to achieve lower bounds relied on powerfree
morphisms. However, the lower bounds obtained in this way are not
particularly good, since they are considerably smaller than the upper
bounds as well as reliable numerical estimates of the actual value of
the entropy. A completely different approach introduced recently by
Kolpakov \cite{Kolpakov:2007}, which amounts to choosing a parameter
value to satisfy a number of inequalities derived from a
Perron-Frobenius-type argument, provides surprisingly good lower
bounds for the entropy of ternary squarefree and binary cubefree
words.

In the following section, we briefly introduce the notation and basic
terminology; see \cite{Lothaire:1997} for a more detailed
introduction. We continue with a summary of results on $k$-powerfree
morphisms, which can be used to derive lower bounds for the
corresponding entropy. We then proceed by introducing the entropy of
$k$-powerfree words and summarise the methods to derive upper and
lower bounds in general, and for binary cubefree and ternary
squarefree words in particular. We conclude with a discussion of the
frequencies of letters in binary cubefree and ternary squarefree
words.

\section{Powerfree words and morphisms}

Define an \emph{alphabet} $\A$ as a finite non-empty set of symbols
called letters. The cardinality of $\A$ is denoted by
$\Card(\A)$. Finite or infinite sequences of elements from $\A$ are
called \emph{words}. The \emph{empty word} is denoted by $\e$. The set
of all finite words, the operation of concatenation of words and the
empty word $\e$ form the free monoid $\A^{*}$. The free
semigroup generated by $\A$ is $\A^{+}:=\A \setminus \{ \e \}$.

The \emph{length} of a word $u\in\A^{*}$, denoted by $|u|$, is the
number of letters that $u$ consists of. The length of the empty word
is $|\e|:=0$.

For two words $u,v \in \A^{*}$, we say that $v$ is a \emph{subword} or
a \emph{factor} of $u$ if there are words $x, y\in \A^{*}$ such that
$u=xvy$. If $x=\e$, the factor $v$ is called a \emph{prefix} of $u$,
and if $y=\e$, $v$ is called a \emph{suffix} of $u$. Given a set of
words $X \subset \A^{*}$ (here and in what follows, the symbol
$\subset$ is meant to include the possibility that both sets are
equal), the set of all factors of words in $X$ is denoted by $\Fact
(X)$.

A map $\rho\! :\A^{*} \rightarrow \B^{*}$, where $\A$ and $\B$ are
alphabets, is called a \emph{morphism} if $\rho (uv)=\rho (u)\rho (v)$
holds for all $u,v \in \A^{*}$. Obviously, a morphism $\rho $ is
completely determined by $\rho (a)$ for all $a\in \A$, and satisfies
$\rho (\e)=\e$. A morphism $\rho\! :\A^{*} \rightarrow \B^{*}$ is
called \emph{$\ell$-uniform}, if $|\rho (a)|=\ell$ for all $a\in \A.$

For a word $u$, we define $u^{0}:=\e$, $u^{1}:=u$ and, for an integer
$k>1$, the power $u^{k}$ as the concatenation of $k$ occurrences of
the word $u$. If $u\neq \e$, $u^{k}$ is called a \emph{$k$-power}. A
word $v$ contains a $k$-power if at least one of its factors is a
$k$-power.  If a word does not contain any $k$-power as a factor, it
is called \emph{$k$-powerfree}. If a word does not contain the
$k$-power of any word up to a certain length $p$ as a factor, it is
called \emph{length-$p$ $k$-powerfree}, i.e., $w=xu^{k}y $ implies
that $u=\e$ whenever $x,u,y \in \A^{*}$ with $|u| \leq p$.  We denote
the set of $k$-powerfree words in an alphabet $\A$ by
$F^{(k)}(\A)\subset \A^{*}$ and the set of length-$p$ $k$-powerfree by
$F^{(k,p)}(\A)\subset \A^{*}$.  By definition, the empty word $\e$ is
$k$-powerfree for all $k$. A word $w\in\A^{*}$ is called
\emph{primitive}, if $w=v^n$, with $v\in\A^{*}$ and $n \in \N$,
implies that $n=1$, meaning that $w$ is not a proper power of another
word $v$.

A morphism $\rho\! : \A^{*} \rightarrow \B^{*}$ is called
\emph{$k$-powerfree}, if $\rho (u)$ is $k$-powerfree for every
$k$-powerfree word $u$. In other words, $\rho $ is powerfree if $\rho
\bigl( F^{(k)}(\A)\bigr)\subset F^{(k)}(\B)$. A \emph{test-set} for
$k$-powerfreeness of morphisms on an alphabet $\A$ is a set $T \subset
\A^{*}$ such that, for any morphism $\rho\! :\A^{*} \rightarrow \B^{*}$,
$\rho $ is $k$-powerfree if and only if $\rho (T)$ is $k$-powerfree. A
morphism is called \emph{powerfree} if it is a $k$-powerfree morphism
for every $k \ge 2$.

In particular, $2$-powerfree and $3$-powerfree words and morphisms are
called \emph{squarefree} and \emph{cubefree}, respectively. A morphism
from $\A^{*}$ to $\B^{*}$ with $\Card (A)=2$ is also called a
\emph{binary} morphism. The notion of powerfreeness can be extended to
non-integer powers; see, for instance, Ref.~\cite{Karhu:2004} for an
investigation of $k$-powerfree binary words for $k\ge 2$. However, in
this article we shall concentrate on the cases $k=2$ and $k=3$, and
hence restrict the discussion to integer powers.

\section{Characterisations of $\boldsymbol{k}$-powerfree morphisms}

In what follows, we summarise a number of relevant results on
$k$-powerfree morphisms. In particular, we are interested in the
question how to test a specified morphism for $k$-powerfreeness. We
start with results relating to the case $k=2$.

\subsection{Characterisations of squarefree morphisms}

A sufficient (but in general not necessary) condition for the
squarefreeness of a morphism is known since 1979.
\begin{theorem}[Bean et al.~\cite{Bean:1979}]\label{thm:bean}
A morphism $\rho\! :\A^{*} \rightarrow \B ^{*}$ is squarefree if
\begin{itemize}
\i $\rho (w)$ is squarefree for every squarefree word
$w \in \A^{*}$ of length\/ $|w|\le 3$;
\ii $a=b$ whenever $a,b \in \A$ and $\rho (a)$ is a factor of $\rho (b)$.
\end{itemize}
\end{theorem}

If the morphism $\rho$ is uniform, this condition is in
fact also necessary, because in this case $\rho(a)$ being a factor of
$\rho(b)$ implies that $\rho (a)=\rho (b)$. If $a,b\in \A$ exist with
$a\neq b$ and $\rho (a)=\rho (b)$, then clearly $\rho $ is not
squarefree since $\rho (ab)=\rho (a)\rho (b)$ is a square. This gives
the following corollary.

\begin{corollary}
A uniform morphism $\rho\! :\A^{*} \rightarrow \B ^{*}$ is squarefree if and
only if $\rho (w)$ is squarefree for every squarefree word $w \in \A^{*}$
of length\/ $|w|\le 3$.
\end{corollary}

This corollary corresponds to Brandenburg's Theorem 2 in
Ref.~\cite{Brandenburg:1983} which only demands that $\rho (w)$ is
squarefree for every squarefree word $w \in \A^{*}$ of length exactly
$3$. A short calculation reveals that this condition is equivalent to
(i), because every squarefree word of length smaller than $3$ occurs
as a factor of a squarefree word of length $3$.

For the next characterisation, we need the notion of a pre-square with
respect to a morphism $\rho $. Let $\A$ be an alphabet, $w \in \A^{*}$ a
squarefree word  and $\rho\! :\A^{*} \rightarrow \B ^{*}$ a morphism. A
factor $u \neq \e$ of $\rho (w)=\alpha u \beta$ is called a
\emph{pre-square} with respect to $\rho $, if there exists a word $w' \in
\A^{*}$ satisfying: $ww'$ is squarefree and $u$ is a prefix of $\beta
\rho (w')$ or $w'w$ is squarefree and $u$ is a suffix of
$\rho (w')\alpha$. Obviously, if $u$ is a pre-square, then either $\rho (ww')$
or $\rho (w'w)$ contains $u^2$ as a factor.

\begin{theorem}[Crochemore \cite{Crochemore:1982}]
 A morphism $\rho\! :\A^{*} \rightarrow \B ^{*}$ is squarefree if and only if
\begin{itemize}
 \i $\rho (w)$ is squarefree for every squarefree word $w \in \A^{*}$
 of length\/ $|w|\le 3$;
 \ii for any $a \in \A$, $\rho (a)$ does not have any internal pre-squares.
\end{itemize}
\end{theorem}

It follows that, for a ternary alphabet $\A$, a finite test-set
exists, as specified in the following corollary. However, the
subsequent theorem shows that, as soon as we consider an alphabet with
$\Card(\A) > 3$, no such finite test-sets exist, so the situation
becomes more complex when considering larger alphabets.

\begin{corollary}[Crochemore \cite{Crochemore:1982}]
Let\/ $\Card(\A)=3$. A morphism $\rho\! :\A^{*} \rightarrow \B^{*}$ is squarefree
if and only if  $\rho (w)$ is squarefree for every squarefree word $w \in \A^{*}$ of
length\/ $|w|\le 5$.
\end{corollary}

\begin{theorem}[Crochemore \cite{Crochemore:1982}]
 Let\/ $\Card(\A) > 3$. For any integer $n$, there exists a morphism
 $\rho\! :\A^{*} \rightarrow \B^{*}$ which is not squarefree, but maps all
 squarefree words of length up to $n$ on squarefree words.
\end{theorem}

\subsection{Characterisations of cubefree and $k$-powerfree morphisms}

We now move on to characterisations of cubefree and $k$-powerfree
morphisms for $k>3$. We start with a recent result on cubefree binary
morphisms.

\begin{theorem} [Richomme, Wlazinski \cite{Richomme:2002}]
A set $T\subset \{a,b\}^{*}$ is a test-set for cubefree morphisms
from $\A^{*}=\{a,b\}^{*}$ to $\B^{*}$ with\/ $\Card(\B) \ge 2$ if and
only if\/ $T$ is cubefree and\/ $\Fact(T)\supset T_{\mathrm{min}}$, where
\[
T_{\mathrm{min}}:=\{abbabba, baabaab, ababba, babaab, abbaba, baabab, aabba,
bbaab, abbaa, baabb, ababa, babab\}.
\]
\end{theorem}

Obviously, the set $T_{\mathrm{min}}$ itself is a test-set for
cubefree binary morphisms. Another test-set is the set of cubefree
words of length $7$, as each word of $T_{\mathrm{min}}$ appears as a
factor of this set. There are even single words which contain all the
elements of $T_{\mathrm{min}}$ as factors. For instance, the cubefree
word $aabbababbabbaabaababaabb$ is one of the $56$ words of length
$24$ which are test-sets for cubefree morphisms on $\{a,b\}$. The
length of this word is optimal: no cube-free word of length $23$
contains all the words of $T_{\mathrm{min}}$ as factors.

The following sufficient characterisation of $k$-powerfree morphisms
generalises Theorem~\ref{thm:bean} to integer powers $k>2$.

\begin{theorem}[Bean et al.~\cite{Bean:1979}]\label{thm:bean2}
Let $\rho\!: \A^{*} \rightarrow \B^{*}$ be a morphism for alphabets $\A$
and $\B$ and let $k>2$. Then $\rho $ is $k$-powerfree if
\begin{itemize}
\i $\rho (w)$ is $k$-powerfree whenever $w \in \A^{*}$ is $k$-powerfree
and of length\/ $|w|\le k+1$;
\ii $a=b$ whenever $a,b \in \A$ with $\rho(a)$ a factor of $\rho (b)$;
\iii the equality $x \rho (a) y=\rho (b)\rho (c)$, with $a,b,c \in \A$ and
$x, y \in \B^{*}$, implies that either $x=\e$, $a=b$ or $y=\e$, $a=c$.
\end{itemize}
\end{theorem}

As in the squarefree case above, a uniform morphism $\rho $ for which (i)
holds also meets (ii), because uniformity implies that $\rho (a)=\rho (b)$. If
$a\neq b$, the word $a^{k-1}b$ is $k$-powerfree but
$\rho (a^{k-1}b)=\rho (a)^k$ is a $k$-power, which produces a
contradiction. The condition (iii) means that, for all letters
$a\in\A$, the images $\rho(a)$ do not occur as an inner factor of
$\rho(bc)$ for any $b,c \in \A$. In general, this is not necessary for
uniform morphisms; an example is given by the Thue-Morse morphism
$\rho$ of Eq.~\eqref{eq:tm}. For instance, $\rho (00)=0101=0\rho (1)1$, which violates
condition (iii) in Theorem~\ref{thm:bean2}. Nevertheless, the Thue-Morse
morphism is cubefree \cite{Thue:1977}.

Alphabets with $\Card (\B)<2$ only provide trivial results, because
the only $k$-powerfree morphism from $\A^{*}$ to $\{\e\}^{*}$ is the
empty morphism $\e$, and for $\Card(\B)=1$ the only additional
morphism is the map for $\Card(\A)=1$ that maps the single element in
$\A$ to the single letter in $\B$. {}From now on, we consider
alphabets with $\Card (\B) \ge 2$. First, we deal with the case
$\Card(\A) \ge 3$.

\begin{theorem}[Richomme, Wlazinski \cite{Richomme:2002}]
 Given two alphabets $\A$ and $\B$ such that\/ $Card(\A) \ge 3$ and\/
 $\Card(\B) \ge 2$, and given any integer $k\ge 3$, there is no finite
 test-set for $k$-powerfree morphisms from $\A^{*}$ to $\B^{*}$.
\end{theorem}

This again is a negative result, which shows that the general
situation is difficult to handle. In general, no finite set of words
suffices to verify the $k$-powerfreeness of a morphism. The situation
improves if we restrict ourselves to uniform morphisms, and look for
test-sets for this restricted class of morphisms only. Here, a
\emph{test-set for $k$-powerfreeness of uniform morphisms} on $\A^{*}$
is a set $T \subset \A ^{*}$ such that, for every uniform morphisms
$\rho $ on $\A^{*}$, $\rho $ is $k$-powerfree if and only if $\rho
(T)$ is $k$-powerfree.

The existence of finite test-sets of this type was recently
established by Richomme and Wlazinski \cite{Richomme:2007}. Let
$\Card(\A) \ge 2$ and $k \ge 3$ be an integer.  Define
\[
T^{(k)}(\A) := U^{(k)}(\A) \cup \bigl(F^{(k)}(\A) \cap V^{(k)}(\A) \bigr)
\]
where $U^{(k)}(\A)$ is the set of $k$-powerfree words over $\A$ of
length at most $k+1$, and $V^{(k)}(\A)$ is the set of words over $\A$
that can be written in the form $a_0 w_1 a_1 w_2\ldots a_{k-1} w_k
a_k$ with letters $a_0, a_1, \ldots , a_k\in\A$ and words $w_1,w_2
\ldots w_k\in\A^{*}$ which contain every letter of $\A$ at most once
and satisfy $\bigl| |w_i|-|w_j|\bigr| \le 1$.  Obviously, this set is
finite and comprises words with a maximum length of
\[
  \max \bigl\{|w| \;\big|\; w \in T^{(k)}(\A) \bigr\}
   \le k \bigl(\Card(\A|)+1\bigr) + 1.
\]

\begin{theorem}[Richomme, Wlazinski \cite{Richomme:2007}]\label{thm:testset}
Let\/ $\Card(\A) \ge 2$ and $k \ge 3$ be an integer.  The finite set\/
$T^{(k)}(\A)$ is a test-set for $k$-powerfreeness of uniform morphisms
on $\A^{*}$.
\end{theorem}

Due to the upper bound on the maximum length of words in
$T^{(k)}(\A)$, the following corollary is immediate.

\begin{corollary}[Richomme, Wlazinski \cite{Richomme:2007}]\label{coro:rw}
A uniform morphism $\rho$ on $\A^{*}$ is $k$-powerfree for an integer power
$k \ge 3$ if and only if $\rho(w)$ is $k$-powerfree for all $k$-powerfree words
$w$ of length at most $k \bigl(\Card(\A)+1\bigr) + 1$.
\end{corollary}

Although this result provides an explicit test-set for
$k$-powerfreeness, it is of limited practical use, simply because the
test-set becomes large very quickly. Already for $\Card(\A)=4$ and
$k=3$, the set $T^{(3)}(\A)$ has $26247020$ elements. For comparison,
the set of cubefree words in four letters of length $16$, as required
in Corollary~\ref{coro:rw}, has $1939267560$ elements, so is still
much larger.

Finally, let us quote the following result of Ker\"{a}nen
\cite{Keraenen:1985}, which characterises $k$-powerfree binary
morphisms and indicates that the test-set of Theorem~\ref{thm:testset}
is far from optimal.

\begin{theorem}[Ker{\"a}nen \cite{Keraenen:1985}]
Let $\rho\! :\{a,b\} \rightarrow \B^{*}$ be a uniform morphism with $\rho (a)
\neq \rho (b)$ and primitive words $\rho (a), \rho (b)$ and $\rho (ab)$. For every
$k$-powerfree word $w \in \{a,b\}^{*}$, $\rho (w)$ is $k$-powerfree if and
only if $\rho (v)$ is $k$-powerfree for every subword $v$ of $w$ with
\[
  |v| \le
  \begin{cases}
  4 & \text{for $3 \le k \le 6$}; \\
  \frac{2}{3}(k+1) & \text{for $k \ge 7$}.
  \end{cases}
\]
\end{theorem}

\section{Entropy of powerfree words}

Let $\A$ be an alphabet. A subset $X \subset \A^{*}$ is called
\emph{factorial} if for any word $x \in X$ all factors of $x$ are also
contained in $X$. Define for a factorial subset $X \subset \A^{*}$ the
number of words of length $n$ occurring in $X$ by $c^{}_{X}(n)$. This
number gives some idea of the complexity of $X$: the larger the number
of words of length $n$, the more diverse or complicated is the
set. That is why $c^{}_{X}\!:\N \rightarrow \N$ is called the
\emph{complexity function} of $X$.

\begin{definition} \label{def:entropy}
The \emph{(combinatorial) entropy} of an infinite factorial set\/
$X \subset \A^{*}$ is defined by
\[
   h(X)=\lim_{n \rightarrow \infty}\frac{1}{n} \log c^{}_{X}(n).
\]
The requirement that $X$ is factorial ensures the existence of the
limit, see for example \cite[Lemma~1]{Baake:1997}.
\end{definition}

We note the following:
\begin{itemize}
\i If $X \subset \A^{*}$ with $\Card(\A)=r$, then  $1 \le c^{}_{X}(n) \le r^n$ for all $n$ which implies
$0 \le h(X) \le \log r$.
\ii If $X=\A^{*}$ with $\Card(\A)=r$, then $c^{}_{X}(n)=r^n$ and $h(X)= \log r$.
\end{itemize}
The set of $k$-powerfree words $F^{(k)}(\A)$ over an alphabet $\A$ is
obviously a factorial subset of $\A^{*}$, which is infinite for
suitable values of $k$ for a given alphabet $\A$. The precise value of
the corresponding entropy, which coincides with the topological
entropy \cite{Walters}, is not known, but lower and upper bounds exist
for many cases. Recently, much improved upper and lower bounds have
been established for $h\bigl(F^{(2)}(\{0,1,2\})\bigr)$ and
$h\bigl(F^{(3)}(\{0,1\})\bigr)$, which will be outlined
below. Generally, it is easier to find upper bounds than to give lower
bounds, due to the factorial nature of the set of $k$-powerfree words,
so we start with describing several methods to produce upper bounds on
the entropy.

\subsection{Upper bounds for the entropy}

A simple way to provide upper bounds is based on the enumeration of the
set of $k$-powerfree words up to some length. Clearly, for the case of
$r=\Card(\A)$ letters, the number of words
$c(n):=c^{}_{F^{(k)}(\A)}(n)$ is bounded by $r^n$, so the
corresponding entropy is $h:=h\bigl(F^{(k)}(\A)\bigr)\le \log r$, as
mentioned above.  Suppose we know the actual value of $c(n)$ for some
fixed $n$. Then, due to the factorial nature of the set
${F^{(k)}(\A)}$,
\[
    c(mn) \le c(n)^m
\]
for any $m\ge 1$. Hence
\begin{equation}\label{eq:enulim}
     h = \lim_{m\to\infty} \frac{\log c(mn)}{mn}
       \le \frac{\log c(n)}{n},
\end{equation}
which, for any $n$, yields an upper bound for $h$. Obviously, the
larger the value of $n$, the better the bound obtained in this way. In
some cases, the bound can be slightly improved by considering words
that overlap in a couple of letters; see \cite{Baake:1997} for an
example.

Sharper upper bounds can be produced by following a different
approach, namely by considering a set of words that do not contain
$k$-powers of a fixed finite set of words, for instance $k$-powers
of all words up to a given length. This limitation means that the
number of forbidden words is finite, and that the resulting factorial
set has a larger entropy than the set of $k$-powerfree words, so the
latter provides an upper bound. Again, by increasing the number of
forbidden words, the bounds can be systematically improved.

As Noonan and Zeilberger pointed out \cite{Noonan:1999}, it is
possible to calculate the generating function for the numbers of words
avoiding a finite set of forbidden words by solving a system of linear
equations. The generating functions are rational functions, and the
location of the pole closest to the origin determines the radius of
convergence, and hence the entropy of the corresponding set of
words. This approach has been applied in Ref.~\cite{Richard:2004} to
derive an upper bound for the set of squarefree words in three
letters, and generating functions for cubefree words in two letters
are discussed below.

A related, though computationally easier approach is based on a
Perron-Frobenius argument. It is sometimes referred to as the
`transfer matrix' or the `cluster' approach. Here, a matrix is
constructed, which determines how $k$-powerfree words of a given
length can be concatenated to form $k$-powerfree words, and the growth
rate is then determined by the maximum eigenvalue of this matrix. Both
methods yield upper bounds that can be improved by increasing the
length of the words involved, and in principle can approximate the
entropy arbitrarily well, though in practice this is limited by the
computational problem of computing the leading eigenvalue of a large
matrix, or solving a large system of linear equations; see, for
instance, \cite{Ochem:2006} for details.

\subsection{Lower bounds for the entropy}

Until very recently, all methods used to prove that the entropy of
$k$-powerfree words is positive and to establish lower bounds on the
entropy were based on $k$-powerfree morphisms. Clearly, a
$k$-powerfree morphism, iterated on a single letter, produces
$k$-powerfree words of increasing length and suffices to show the
existence of infinite $k$-powerfree words. For example, the fact that
the Thue-Morse morphism \eqref{eq:tm} is cubefree shows the existence
of cubefree words of arbitrary length in two letters. To prove that
the entropy is actually positive, one has to show that the number of
$k$-powerfree words grows exponentially with their length.
Essentially, this is achieved by considering $k$-powerfree morphisms
from a larger alphabet. The following theorem is a generalisation of
Brandenburg's method, compare \cite{Brandenburg:1983}, and provides a
path to produce lower bounds for the entropy of $k$-powerfree words.

\begin{theorem}\label{thm:gbt}
Let $\A$ and $\B$ be alphabets with\/ $\Card(\A)=r\Card(\B)$, where
$r>1$ is an integer. If there exists an\/ $\ell$-uniform\/ $k$-powerfree
morphism $\rho\!:\A^{*} \rightarrow \B^{*}$, then
\[
    h\bigl(F^{(k)}(\B)\bigr) \ge \frac{\log r}{\ell-1}.
\]
\end{theorem}
\begin{proof}
For this proof define $h:=h\bigl(F^{(k)}(\B)\bigr)$,
$c(n):=c_{F^{(k)}(\B)}(n)$ and $s:=\Card(\B)$. Label the elements of
$\A$ as $\{a_{11}, \ldots, a_{1r}, a_{21}, \ldots, a_{2r}, \ldots,
a_{s1},\ldots,a_{sr} \}$ and the elements of $\B$ as $\{b_1, \dots,
b_s\}$. Define the map $\phi\!:\A^{*} \rightarrow \B^{*}$ as
$\phi(a_{ij}):=b_i$ for $i=1, \ldots, s$ and $j=1, \ldots, r$. Hence
$\Card(\phi^{-1}(b_i))=r$. Every $k$-powerfree word of length $m$ over
$\B$ has $r^{m}$ different preimages of $\phi$ which, by construction,
consist only of $k$-powerfree words. These words are mapped by $\rho$,
which is injective due to its $k$-powerfreeness, to different
$k$-powerfree words of length $m\ell$ over $\B$. This implies the
inequality
\begin{equation}\label{eq:submul}
  c(m\ell) \ge r^{m}c(m)
\end{equation}
for any $m>0$. This means
\[
  \left( \frac{c(m \ell)}{c(m)}\right)^{\frac{1}{m}}\ge r,
\]
and hence
\[
\ell\,\frac{\log c(m\ell)}{m\ell} - \frac{\log c(m)}{m} \ge \log r
\]
for any $m>0$. Taking the limit as $m\to\infty$ gives
\[
   (\ell-1) h \ge \log r,
\]
thus establishing the lower bound.
\end{proof}

This result means that, whenever we can find a uniform $k$-powerfree
morphism from a sufficiently large alphabet, it provides a lower bound
for the entropy. Clearly, the larger $r$ and the smaller $\ell$ the
better the bound, so one is particularly interested in uniform
$k$-powerfree morphisms from large alphabets of minimal length.

Another method due to Brinkhuis \cite{Brinkhuis:1983}, which is
related to Brandenburg's method, can be generalised as follows.  Let again
$\B=\{b_1, \ldots, b_s\}$ be an alphabet and $r \in \N$. For $i=1,
\ldots s$ let
\[
  \U_i:=\{U_{i,1}, U_{i,2}, \ldots, U_{i,r}\}
\]
with $U_{i,j} \subset F_{\ell}^{(k)}(\B)$, where the latter denotes
the words in $F^{(k)}(\B)$ which have length $\ell$. The set
$\U=\{\U_1, \ldots \U_{n}\}$ is called an
\emph{$(k,\ell,r)$-Brinkhuis-set} if the $\ell$-uniform substitution
(in the context of formal language theory) $\phi$ from $\B^{*}$ to
itself defined by
\[
   \phi\!: b_i \mapsto  \U_i \; \text{ for } \; i=1, \ldots, s
\]
has the property $\phi(F^{(k)}(\B)) \subset F^{(k)}(\B)$. In other
words $\U$ is an $(k, \ell, r)$-Brinkhuis-set if the substitution of
every letter $b_i$, occurring in a $k$-powerfree word, by an element of
$\U_i$ results in a $k$-powerfree word over $\B$.  The existence of a
$(k, \ell, r)$-Brinkhuis-set delivers the lower bound
\[
   h(F^{(k)}(\B)) \ge \frac{\log r}{\ell-1}
\]
because every $k$-powerfree word of length $m$ is mapped to $r^{m}$
powerfree words of length $\ell m$; compare Eq.~\eqref{eq:submul}.

The method of Brinkhuis is stronger than the method of
Brandenburg. Not every $(k, \ell, r)$-Brinkhuis-set implies a map
according to Theorem \ref{thm:gbt}, see \cite[p.~287]{Berstel:2005}
for an example. Conversely, if there exists an $\ell$-uniform
$k$-powerfree morphism $\rho\!:\A^{*} \rightarrow \B^{*}$ according to
Theorem \ref{thm:gbt}, then there exists a $(k, \ell,
r)$-Brinkhuis-set, namely
$\U_{i}=\bigl\{\rho(a_{i1}),\ldots,\rho(a_{ir})\bigr\}$ for
$i=1,\ldots,s$, with the notation of Theorem~\ref{thm:gbt}.

Brinkhuis' method was applied in
Refs.~\cite{Elser:1983,Grimm:2001,Sun:2003}; see also below for a
summary of bounds obtained for binary cubefree and ternary squarefree
words. These bounds have in common that they are nowhere near the
actual value of the entropy, and while a systematic improvement is
possible by increasing the value of $r$ in Theorem~\ref{thm:gbt}
(which, however, also means that one has to consider larger values of
$\ell$), it will always result in a much smaller growth rate, because
only a subset of words is obtained in this way.

Recently, a different approach has been proposed \cite{Kolpakov:2007},
based essentially on the derivation of an inequality
\begin{equation}
    S_m(n+1) \geq \alpha S_m(n)
\end{equation}
for the weighted sum $S_m(n)$ of the number of elements in a certain
subset (which depends on the choice of $m\in\N$) of squarefree (resp.\
cubefree) words of length $n$ over a ternary (resp.\ binary) alphabet
and a parameter $\alpha >1$ which satisfies two inequalities for $i=m,
m+1, \ldots, n-1$. The estimation of $S_m(n+1)$ starts from a
Perron-Frobenius argument and concludes with the observation that the
order of growth of the number of squarefree (resp.\ cubefree) words
cannot be less than the order of growth of $S_m(n)$, which is
$\alpha$. This implies
\[
    h(F^{(k)}(\A))> \log(\alpha)
\]
for $k=2$ and $\A=\{0,1,2\}$ or $k=3$ and $\A=\{0,1\}$, with the
corresponding values for $\alpha$. In the end, this method leads to a
recipe to check, with a computer, several conditions for the
parameters (including $m$), which ensure that the inequality for $S_m$
holds. By increasing the parameter $m$, it appears to be possible to
estimate the growth rate of cubefree and squarefree words with an
arbitrary precision. For details, we refer the reader to
Ref.~\cite{Kolpakov:2007}.

\section{Bounds on the entropy of binary cubefree and ternary squarefree words}

We now consider the two main examples, binary cubefree and ternary
squarefree words, in more detail, reviewing the bounds derived by the
various approaches mentioned above. We start with the discussion of
binary cubefree words, and then give a brief summary of the analogous
results for ternary squarefree words.

\subsection{Binary cubefree words}

Define for this section $b(n):=c^{}_{F^{(3)}(\{0,1\})}(n)$ as the
number of binary cubefree words of length $n$ and
$h:=h\bigl(F^{(3)}(\{0,1\})\bigr)$ as the entropy of cubefree words
over the alphabet $\{0,1\}$. The values for $b(n)$ with $n\le 47$ are
given in \cite{Edlin:1999}; an extended list for $n\le 80$ is shown in
Table~\ref{tab:cfw}. They were obtained by a straight-forward
iterative construction of cubefree words, appending a single letter at
a time. According to Eq.~\eqref{eq:enulim}, the corresponding upper
limit for the entropy $h$ is
\[
    h \le \frac{\log b(80)}{80} \simeq 0.389855 .
\]
For comparison, the limit obtained using the number of words of length
$79$ is $0.390020$, which indicates that these limits are still
considerably larger than the actual value of $h$. As in the case of
ternary squarefree words \cite{Richard:2004}, the asymptotic behaviour
of $b(n)$ fits a simple form $b(n)\sim A x_{\text{c}}^{-n}$ as
$n\rightarrow\infty$, pointing at a simple pole as the dominating
singularity of the corresponding generating function at
$x=x_{\text{c}}$. The estimated values of the coefficients are
$A\simeq 2.847$ and $x_{\text{c}}\simeq 1.4575773$, leading to a
numerical estimate of $h=\log(x_{\text{c}})\simeq 0.3767757$ for the
entropy.

Let us compare this with the upper limit derived from generating
function of the number of binary length-$p$ cubefree words. To this
end, let $b_{p}(n):=c^{}_{F^{(3,p)}(\{0,1\})}(n)$ denote the number of
length-$p$ cubefree words, and define
\begin{equation}\label{eq:genfun}
   B_{p}(x)=\sum^{\infty}_{n=0}b_{p}(n)\, x^n
\end{equation}
to be the generating function for the number of binary length-$p$
cubefree words. These functions of $x$ are rational \cite{Noonan:1999}.
The first few generating functions read
\begin{eqnarray*}
   B^{}_{0}(x) \; = & \frac{1}{1-2x} & =
               \; 1 + 2x + 4x^2 + 8 x^3 + 16 x^4 + 32 x^5 + 64 x^6 + \ldots , \\
   B^{}_{1}(x) \; = & \frac{1+x+x^2}{1-x-x^2} & =
               \; 1 + 2x + 4x^2 + 6 x^3 + 10 x^4 + 16 x^5 + 26 x^6 + \ldots , \\
   B^{}_{2}(x) \; = & \frac{1+2x+3x^2+3x^3+3x^4+3x^5+2x^6}{1-x^2-x^3-x^4-x^5}
            & =\;  1 + 2x + 4x^2 + 6 x^3 + 10 x^4 + 16 x^5 + 24 x^6 + \ldots
\end{eqnarray*}
The degrees of the numerator and denominator polynomials for $p\le 14$
are given in Table~\ref{tab:genfun}. The generating functions
$B^{}_{p}(x)$ have a finite radius of convergence, determined by the
location of the zero $x^{}_{\mathrm{c}}$ of its denominator polynomial
which lies closest to the origin. A plot of the location of poles of
$B^{}_{14}(x)$ is shown in Figure~\ref{fig:poles}. It very much
resembles the analogous distribution for ternary squarefree words
\cite{Richard:2004}; again, the poles seem to accumulate, with
increasing $p$, on or near the unit circle, which may indicate the
presence of a natural boundary beyond which the generating function
for cubefree binary words (corresponding to taking $p\to\infty$)
cannot be analytically continued; see \cite{Richard:2004} for a
discussion of this phenomenon in the case of ternary squarefree
words.

\begin{table}
\caption{The number $b(n)$ of binary cube-free words of length $n$
for $n\le 80$.}\label{tab:cfw}
\begin{center}
\begin{tabular}{rrp{2ex}rrp{2ex}rrp{2ex}rr}
\hline
\multicolumn{1}{c}{$n$} &
\multicolumn{1}{c}{$b(n)$} & &
\multicolumn{1}{c}{$n$} &
\multicolumn{1}{c}{$b(n)$} & &
\multicolumn{1}{c}{$n$} &
\multicolumn{1}{c}{$b(n)$} & &
\multicolumn{1}{c}{$n$} &
\multicolumn{1}{c}{$b(n)$}\\
\hline
   $1$ &    $2$ && $21$ &    $7754$ && $41$ &    $14565048$ && $61$ &    $27286212876$\\
   $2$ &    $4$ && $22$ &   $11320$ && $42$ &    $21229606$ && $62$ &    $39771765144$\\
   $3$ &    $6$ && $23$ &   $16502$ && $43$ &    $30943516$ && $63$ &    $57970429078$\\
   $4$ &   $10$ && $24$ &   $24054$ && $44$ &    $45102942$ && $64$ &    $84496383550$\\
   $5$ &   $16$ && $25$ &   $35058$ && $45$ &    $65741224$ && $65$ &   $123160009324$\\
   $6$ &   $24$ && $26$ &   $51144$ && $46$ &    $95822908$ && $66$ &   $179515213688$\\
   $7$ &   $36$ && $27$ &   $74540$ && $47$ &   $139669094$ && $67$ &   $261657313212$\\
   $8$ &   $56$ && $28$ &  $108664$ && $48$ &   $203577756$ && $68$ &   $381385767316$\\
   $9$ &   $80$ && $29$ &  $158372$ && $49$ &   $296731624$ && $69$ &   $555899236430$\\
  $10$ &  $118$ && $30$ &  $230800$ && $50$ &   $432509818$ && $70$ &   $810266077890$\\
  $11$ &  $174$ && $31$ &  $336480$ && $51$ &   $630416412$ && $71$ &  $1181025420772$\\
  $12$ &  $254$ && $32$ &  $490458$ && $52$ &   $918879170$ && $72$ &  $1721435861086$\\
  $13$ &  $378$ && $33$ &  $714856$ && $53$ &  $1339338164$ && $73$ &  $2509125828902$\\
  $14$ &  $554$ && $34$ & $1041910$ && $54$ &  $1952190408$ && $74$ &  $3657244826158$\\
  $15$ &  $802$ && $35$ & $1518840$ && $55$ &  $2845468908$ && $75$ &  $5330716904964$\\
  $16$ & $1168$ && $36$ & $2213868$ && $56$ &  $4147490274$ && $76$ &  $7769931925578$\\
  $17$ & $1716$ && $37$ & $3226896$ && $57$ &  $6045283704$ && $77$ & $11325276352154$\\
  $18$ & $2502$ && $38$ & $4703372$ && $58$ &  $8811472958$ && $78$ & $16507465616784$\\
  $19$ & $3650$ && $39$ & $6855388$ && $59$ & $12843405058$ && $79$ & $24060906866922$\\
  $20$ & $5324$ && $40$ & $9992596$ && $60$ & $18720255398$ && $80$ & $35070631260904$\\
\hline
\end{tabular}
\end{center}
\end{table}

As a consequence of Pringsheim's theorem
\cite[Sec.~7.2]{Titchmarsh:1976}, there is a dominant singularity on
the positive real axis; we denote the position of the singularity by
$x^{}_{\mathrm{c}}$. For the cases we considered, this simple pole
appears to be the only dominant singularity. Since the radius of
convergence of the power series $B^{}_{p}(x)$ is given by
$\bigl(\limsup_{n\to\infty} \sqrt[n]{b^{}_{p}(n)}\,\bigr)^{-1}$, the
entropy $h_{p}$ of the set of binary length-$p$ cubefree words is
$h_{p}=-\log x_{\mathrm{c}}$.  Clearly, $h_{p}\ge h_{p'}$ for $p\le
p'$, and $h=\lim_{p\to\infty}h_{p}$, so for any finite $p$ the entropy
$h_{p}$ provides an upper bound of the entropy $h$ of binary cubefree
words. The values of the entropy $h_{p}$ for $p\le 14$ are given in
Table~\ref{tab:genfun}. As was observed for ternary squarefree words
\cite{Richard:2004}, the values appear to converge very quickly with
increasing $p$, but it is difficult to extract a reliable estimate of
the true value of the entropy without making assumptions on the
asymptotic behaviour.

Already in 1983, Brandenburg
\cite{Brandenburg:1983} showed that
\[
   2 \cdot 2^{\frac{n}{9}} \le b(n)
    \le 2 \cdot 1251^{\frac{n}{17}}
\]
which leads in our setting to $0.07701 \le h \le 0.41952$.  The
currently best upper bounds are due to Edlin \cite{Edlin:1999} and
Ochem and Reix \cite{Ochem:2006}.  Analysing length-$15$ cubefree
words up to a finite length, Edlin \cite{Edlin:1999} arrives at the
bound of $h \le 0.376777$ (which is what we would expect to find if we
extended Table~\ref{tab:genfun} to $n=15$, but this would require huge
computational effort to compute the corresponding generating function
completely), while using the transfer matrix (or cluster) approach
described above, Ochem and Reix obtained an upper bound on the growth
rate of $1.45758131$, which corresponds to the bound
\[
     h \le 0.3767784
\]
on the entropy.

\begin{table}
\caption{The entropy $h_{p}$ of binary length-$p$ cubefree words,
obtained from the radius of convergence of the generating functions
$B_{p}(x)$ of Eq.~\eqref{eq:genfun}. Here, $d_{\mathrm{num}}$ and
$d_{\mathrm{den}}$ denote the degree of the polynomial in the
numerator and denominator of $B_{p}(x)$,
respectively.}\label{tab:genfun}
\begin{center}
\begin{tabular}{rrrr}
\hline
\multicolumn{1}{c}{$p$} &
\multicolumn{1}{c}{$d_{\mathrm{num}}$} &
\multicolumn{1}{c}{$d_{\mathrm{den}}$} &
\multicolumn{1}{c}{$h_{p}$} \\
\hline
 0 &    0 &    1 & 0.693147\\
 1 &    2 &    2 & 0.481212\\
 2 &    6 &    5 & 0.427982\\
 3 &   21 &   13 & 0.394948\\
 4 &   29 &   17 & 0.385103\\
 5 &   43 &   25 & 0.380594\\
 6 &   85 &   57 & 0.378213\\
 7 &  127 &   99 & 0.377332\\
 8 &  165 &  127 & 0.377179\\
 9 &  300 &  254 & 0.376890\\
10 &  450 &  395 & 0.376835\\
11 &  569 &  513 & 0.376811\\
12 & 1098 & 1031 & 0.376790\\
13 & 1750 & 1656 & 0.376783\\
14 & 2627 & 2540 & 0.376779\\
\hline
\end{tabular}
\end{center}
\end{table}

\begin{figure}
\caption{Location of poles of the generating function $B^{}_{14}(x)$.}
\centerline{\includegraphics[width=0.6\textwidth]{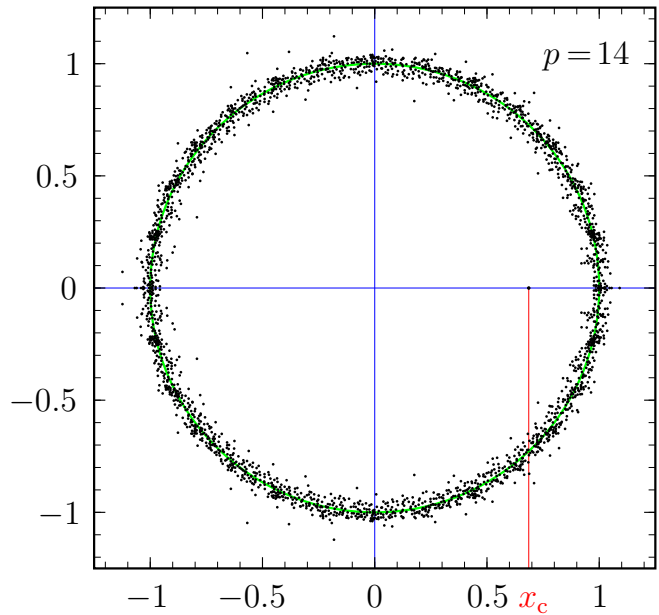}}
\label{fig:poles}
\end{figure}

We now move on to the lower bound and cubefree morphisms. We already
have seen one example above, the Thue-Morse morphism, which is a
cubefree morphism from a binary alphabet to a binary alphabet. As
explained above, it is also useful to find uniform cubefree morphisms
from larger alphabets, because these provide lower bounds on the
entropy. Clearly, if we have a uniform cubefree morphism $\rho\!:
\A^{*}\rightarrow \{0,1\}^{*}$ of length $\ell$, with $\Card(\A)=r$, it
is completely specified by the $r$ words $w^{}_{i}$, $i=1,\ldots,r$,
which are the images of the letters in $\A$. Since any permutation of
the letters in $\A$ will again yield a uniform cubefree morphism, the
set $\bigl\{w^{}_{1},\ldots,w^{}_{r}\bigr\}\subset\{0,1\}^{\ell}$ of
generating words determines the morphism up to permutation of the
letters in $\A$.

Moreover, the set
$\bigl\{\overline{w^{}_{1}},\ldots,\overline{w^{}_{r}}\bigr\}$, where
$\overline{w}$ denotes the image of $w$ under the permutation $0
\leftrightarrow 1$, also defines cubefree morphisms, as does
$\bigl\{\widetilde{w^{}_{1}},\ldots,\widetilde{w^{}_{r}}\bigr\}$,
where $\widetilde{w}$ denotes the reversal of $w$, i.e., the words $w$
read backwards. This is obvious because the test-sets of
Theorem~\ref{thm:testset} are invariant under these operations. Unless
the words are palindromic (which means that $\widetilde{w}=w$), the
set $\bigl\{w^{}_{1},\ldots,w^{}_{r}\bigr\}$ thus represents four
different morphisms (not taking into account permutation of letters in
$\A$), the forth obtained by performing both operations, yielding
$\bigl\{\widetilde{\overline{w^{}_{1}}},\ldots,\widetilde{\overline{w^{}_{r}}}\bigr\}$.

For cubefree morphisms from a three-letter alphabet $\A$ to two
letters one needs words of length at least six. For length six, there
are twelve in-equivalent (with respect to the permutation of letters in
$\A$) cubefree morphisms. The corresponding sets of generating words
are
\[
   \{w^{}_{1},w^{}_{2},w^{}_{4}\},\quad
   \{w^{}_{2},\overline{w^{}_{3}},\widetilde{\overline{w^{}_{3}}}\},\quad
   \{w^{}_{2},\overline{w^{}_{3}},w^{}_{4}\},
\]
and the corresponding
images under the two operations explained above. Here, the four words
are
\[
   w_{1} = 001011,\quad
   w_{2} = 001101,\quad
   w_{3} = 010110,\quad
   w_{4} = 011001.
\]
It turns out that none of these morphisms actually satisfy the sufficient criterion
of Theorem~\ref{thm:bean2}, but cubefreeness was verified using the test-set of
Theorem~\ref{thm:testset}.

One has to go to length nine to find cubefree morphisms from four to two letters.
There are $16$ in-equivalent morphisms with respect to permutations
of the four letters. Explicitly, they are given by the generating sets
\begin{equation}\label{eq:mor9}
 \{w_{1}, w_{2}, \widetilde{\overline{w_{2}}}, \widetilde{\overline{w_{3}}}\},\;
 \{w_{4}, \overline{w_{6}}, \overline{w_{7}}, \overline{w_{9}}\},\;
 \{w_{5}, \overline{w_{5}},  w_{8}, \overline{w_{8}}\},\;
 \{w_{5}, \overline{w_{5}}, \widetilde{w_{8}}, \widetilde{\overline{w_{8}}}\},\;
 \{\overline{w_{6}}, \widetilde{w_{7}},\widetilde{\overline{w_{8}}}, w_{9}\}
\end{equation}
with words
\begin{align*}
w_{1} & = 001001101,&
w_{2} & = 001010011,&
w_{3} & = 001011001,\\
w_{4} & = 001101001,&
w_{5} & = 010010110,&
w_{6} & = 010011010,\\
w_{7} & = 010100110,&
w_{8} & = 011001001,&
w_{9} & = 011010110.
\end{align*}
Note that $w^{}_{9}=\widetilde{w^{}_{9}}$ is a palindrome, and that
two of the five sets are invariant under the permutation
$0\leftrightarrow 1$, which explains why they only represent $16$
different morphisms.

Beyond four letters, the test-set of Theorem~\ref{thm:testset} becomes
unwieldy, but the sufficient criterion of Theorem~\ref{thm:bean2} can
be used to obtain morphisms. However, these may not have the optimal
length, as the examples here show -- again for length nine all
morphisms violate the conditions of Theorem~\ref{thm:bean2}. Still,
this need not be the case; for instance, morphisms from a five-letter
alphabet that satisfy the sufficient criterion exist for length $12$,
which in this case is the optimal length.

As a consequence of Theorem~\ref{thm:gbt}, the morphisms
\eqref{eq:mor9} from a four letter alphabet show that the entropy of
cubefree binary words is positive, and that
\[
  h \ge \frac{\log 2}{8} \simeq 0.08664.
\]
Using the sufficient condition, this bound can be improved. For instance,
for length $15$, one can find cubefree morphisms from $10$ letters,
which yields a lower bound of
\[
  h \ge \frac{\log 5}{14} \simeq 0.11496.
\]
However, a large step to close the gap between these lower bounds and
the upper bound was achieved by the work of Kolpakov \cite{Kolpakov:2007}.
With his approach, a lower bound of
\[
     h \ge 0.37676,
\]
which is the best lower bound so far, has been established. The
difference between this bound and the upper bound $0.3767784$ by Ochem
and Reix \cite{Ochem:2006} is just $10^{-5}$, showing the huge
improvement over the previously available estimates.

\subsection{Ternary squarefree words}

Denote by $a(n):=c^{}_{F^{(2)}(\{0,1,2\})}(n)$ the number of ternary
squarefree words and by $a_{p}(n)$ the number of length-$p$ squarefree
words of length $n$. For this section let
$h:=h\bigl(F^{(2)}(\{0,1,2\})\bigr)$ be the entropy of squarefree words
over the alphabet $\{0,1,2\}$.  See \cite{Baake:1997} for a list of
$a(n)$ for $n\leq 90$ and \cite{Grimm:2001} for $91 \leq n \leq
110$. The generating functions are defined according to the binary
cubefree case. The first four of them are stated in
\cite[Sec.~3]{Richard:2004}, which also contains a list of their radii
of convergence for $p\le 24$.  Already in 1983 Brandenburg
\cite{Brandenburg:1983} showed that
\[
   6 \cdot 2^{\frac{n}{22}} \le a(n) \le 6 \cdot 1172^{\frac{n}{22}}
\]
which leads in our setting to

\[
   0.03151 \le h \le 0.32120.
\]
In 1999, Noonan and Zeilberger \cite{Noonan:1999} lowered the upper
bound to $0.26391$ by means of generating functions for the number of
words avoiding squares of up to length $23$. Grimm and Richard
\cite{Richard:2004} used the same method to improve the upper bound to
$0.263855$. At the moment, the best known upper bound is $0.263740$
which was established by Ochem in 2006 using an approach based on the
transfer matrix (or cluster) method, see \cite{Ochem:2006} for details.

In 1998, Zeilberger showed that a Brinkhuis pair of length $18$
exists, which by Theorem~\ref{thm:gbt} implies that the entropy is
bounded by $h \ge \log(2)/17\simeq 0.04077$ \cite{Ekhad:1998}.  By
going to larger alphabets, this was subsequently improved to $h\ge
\log(65)/40\simeq 0.10436$ by Grimm \cite{Grimm:2001} and $h\ge
\log(110)/42\simeq 0.11192$ by Sun \cite{Sun:2003}. Again, the recent
work of Kolpakov \cite{Kolpakov:2007} has made a large difference to
the lower bounds; he achieved the best current lower bound which is $h
\ge 0.26369$. The difference between the best known upper and lower
bound is now just $5\times 10^{-5}$.

\section{Letter frequencies}

For a finite word $w$ of length $n$, the frequency of the letter $a$
is $\#_a(w)/n\in [0,1]$, where $\#_a(w)$ denotes the number of occurrences of
the letter $a$ in $w$.  In general, infinite $k$-powerfree words need
not have well-defined letter frequencies. However, we can define upper
and lower frequencies $f_a^{+}\ge f_a^{-}$ of a letter $a\in\A$ of a
word $w\in\A^{*}$ as
\[
   f_a^{+}:=\sup_{\{w_{n}\}}\limsup_{n\to\infty}\frac{\#_a(w_n)}{n},\qquad
   f_a^{-}:=\inf_{\{w_{n}\}}\liminf_{n\to\infty}\frac{\#_a(w_n)}{n},
\]
where $w_{n}$ is a $n$-letter subword of $w$.  Here, we take the
supremum and infimum over all sequences $\{w_{n}\}$. Alternatively, we
can compute these frequencies from $a_n^{+}=\max_{w_n\subset w}
\#_a(w_n)$ and $a_n^{-}=\min_{w_n\subset w} \#_a(w_n)$ by
$f_a^{\pm}=\lim_{n\to\infty}a_{n}^{\pm}/n$. The limits exist due to
the subadditivity of the sequences $\{a_n^{+}\}$ and
$\{1-a_{n}^{-}\}$.  If the infinite word $w$ is such that
$f_a^{+}=f_{a}^{-}=:f_{a}$, we call $f_{a}$ the \emph{frequency} of
the letter $a$ in $w$.

The requirement that a word is $k$-powerfree for some $k$ restricts
the possible letter frequencies. For instance, for cubefree binary
words, there cannot be three consecutive zeros, and hence the
frequency of the letter $0$ is certainly bounded from above by $2/3$.
Due to symmetry under permutation of letters, it is bounded from below
by $1/3$. In a similar way, considering maximum and minimum
frequencies of letters in \emph{finite} $k$-powerfree words produces
bounds on the possible (upper and lower) frequencies of letters in
infinite words. It is of interest, for which frequency of a letter
$k$-powerfree words cease to exist, and how the entropy of
$k$-powerfree words depends on the letter frequency. To answer these
questions, $k$-powerfree morphisms are exploited once again, and in
two ways. Firstly, the argument using frequencies in finite words only
produces `negative' results, in the sense that you can exclude the
existence of $k$-powerfree words for certain ranges of the
frequency. To show that $k$-powerfree words of a certain frequency
actually exist, these are produced as fixed points of $k$-powerfree
morphisms. The letter frequency for an infinite word obtained as a
fixed point of a morphism $\rho$ on the alphabet
$\A=\{a_1,a_2,\ldots,a_{m}\}$ is well-defined, and obtained from the
(statistically normalised) right Perron-Frobenius eigenvector of the
associated $m\times m$ substitution matrix $M$ with elements
$M_{ij}=\#_{a_i}\rho(a_j)$; see for instance \cite{BGJ}. For example,
for the Thue-Morse morphism \eqref{eq:tm}, the substitution matrix is
$M = \left(\begin{smallmatrix} 1 & 1 \\ 1 & 1
\end{smallmatrix}\right)$ with Perron-Frobenius eigenvalue $2$ and
corresponding eigenvector $(\frac{1}{2},\frac{1}{2})^{T}$, so both
letters occur with frequency $1/2$ in the infinite Thue-Morse word.

To show that there exist exponentially many words with a given letter
frequency, or, in other words, that the entropy of the set of
$k$-powerfree words with a given letter frequency is positive, a variant
of Theorem~\ref{thm:gbt} is used.

\begin{theorem}\label{thm:gbtf}
Let $\A=\{a_{11},\ldots,a_{1r},a_{21},\ldots,a_{2r},\ldots,
a_{s1},\ldots,a_{sr}\}$ and $\B=\{b_{1},\ldots,b_{s}\}$ be
alphabets with\/ $\Card(\A)=rs$ and\/ $\Card(\B)=s$, where $r,s>1$ are
integers. Assume that there exists an\/ $\ell$-uniform\/ $k$-powerfree morphism
$\rho\!:\A^{*} \rightarrow \B^{*}$ with
\[
   \#_{b}\rho(a_{ij})= \#_{b}\rho(a_{ij'})
\]
for all\/ $b\in\B$, $1\le i\le s$ and\/ $1\le j,j'\le r$. Define the
$r\times r$ matrix $M$ with elements
\[
   M_{ij} =  \#_{b_{i}}\rho(a_{j1}),
\]
and denote its right Perron-Frobenius eigenvector (with eigenvalue
$\ell$) by\/ $(f_{1},\ldots,f_{r})^{T}$, with statistical normalisation\/
$f_{1}+\ldots+f_{r}=1$. Then, the entropy\/ $h$ of the set of\/
$k$-powerfree words in $\B$ with prescribed letter frequencies $f_{i}$
of\/ $b_{i}$, $1\le i\le r$, is bounded by
\[
    h \ge \frac{\log r}{\ell-1}.
\]
\end{theorem}
\begin{proof}
The bound is the same as in Theorem~\ref{thm:gbt}, and the statement
thus follows by showing that the infinite words obtained from the
uniform $k$-powerfree morphism $\rho$ have letter frequency given by
$f_{1},\ldots,f_{r}$.

We again introduce the morphism $\phi\!:\A^{*} \rightarrow \B^{*}$ by
$\phi(a_{ij}):=b_i$ for $i=1, \ldots, s$ and $j=1, \ldots, r$. Every
$k$-powerfree word of length $m$ over $\B$ has $r^{m}$ different
preimages of $\phi$ which, by construction, consist only of
$k$-powerfree words. These words are mapped by $\rho$, which is
injective due to its $k$-powerfreeness, to different $k$-powerfree
words of length $m\ell$ over $\B$. Due to the condition
$\#_{b}\rho(a_{ij})= \#_{b}\rho(a_{ij'})$ on $\rho$, the letter
statistics do not depend on the choice of the preimage under $\phi$.
The letter frequencies of words obtained by the procedure described in
the proof of Theorem~\ref{thm:gbt} are thus well defined, and given by
the right Perron-Frobenius eigenvector of the $r\times r$ matrix $M$.
\end{proof}

Some results for binary cubefree words, as well as a discussion of the
empirical frequency distribution of cubefree binary words obtained
from the enumeration up to length $80$, are detailed below.

\subsection{Binary cubefree words}

When counting the numbers $b(n)$ of binary cubefree words of
length $n$ shown in Table~\ref{tab:cfw}, we also counted the number
$b(n,n^{}_{0})$ of words with $n^{}_{0}$ occurrences of the letter $0$.
Clearly, these numbers satisfy
\[
    b(n) = \sum_{n^{}_{0}=0}^{n} b(n,n^{}_{0})
\]
and $b(n,n-n^{}_{0})=b(n,n^{}_{0})$ as a consequence of the symmetry
under permutation of letters.
Their values for $n=80$ are given in
Table~\ref{tab:e80}.

\begin{table}
\caption{The number of the binary cube-free words of length $80$
with given excess $e=n^{}_{0}-40$ of the letter $0$.}
\label{tab:e80}
\begin{center}
\begin{tabular}{rr}
\hline
\multicolumn{1}{c}{$|e|$} &
\multicolumn{1}{c}{$b(80,40+e)$} \\
\hline
$0\quad$ & $9502419002570$ \\
$1\quad$ & $7575510051076$ \\
$2\quad$ & $3805516412947$ \\
$3\quad$ & $1172047753336$ \\
$4\quad$ & $210113470848$ \\
$5\quad$ & $20038955440$ \\
$6\quad$ & $866998237$ \\
$7\quad$ & $12460464$ \\
$8\quad$ & $26819$ \\
$\ge 9\quad$ & $ 0 $ \\
\hline
\end{tabular}
\end{center}
\end{table}

Obviously, there are at least $32$ and at most $48$ occurrences of the
letter $0$ in any cubefree binary word of length $80$, so the
frequency of a letter is bounded by $2/5\le f_{0}\le 3/5$.  A stronger
bound has been obtained by Ochem \cite{Ochem:2007a}, who showed
(amongst many results for a number of rational powers) that
$f_{0}>\frac{115}{283}\simeq 0.40636$, using a backtracking algorithm.

One is interested to locate the minimum frequency $f_{\text{min}}$,
such that infinite cubefree words with frequency
$f_{0}=f_{\text{min}}$ exist, but not for any
$f_{0}<f_{\text{min}}$. Clearly, the lower bound above is a lower
bound for $f_{\text{min}}$. In order to obtain an upper bound, we need
to prove the existence of an infinite binary cubefree words of a given
letter frequency.  This is again done by using a cubefree morphism,
which provides an infinite word with well-defined letter
frequencies. For instance,
\begin{align*}
    0&\mapsto  011011010110110011011010110 \\
    1&\mapsto  011011010110110011010110110
\end{align*}
is a uniform morphism of length $27$ with substitution matrix
$\left(\begin{smallmatrix}11 & 11 \\ 16 & 16\end{smallmatrix}\right)$,
so the infinite fixed point word has letter frequencies
$f_{0}=\frac{11}{27}$ and $f_{1}=\frac{16}{27}$.  Hence we deduce that
\[
  0.406360\simeq \frac{115}{283} < f_{\text{min}} \le
  \frac{11}{27} \simeq  0.407407.
\]

Using the data from our enumeration of binary cubefree words up to
length $80$, we can study the empirical distribution for small length,
and try to conjecture the behaviour for large
words. Figure~\ref{fig:dist} shows a plot of the normalised data
$b(80,40+e)/b(80)$ of Table~\ref{tab:e80}, compared with a Gaussian
distribution, which appears to fit the data very well. Here, the
Gaussian profile was determined from the variance $\sigma^{2}$
of the data points,
which is approximately $\sigma^{2}\simeq 2.124$.

To draw any conclusions on the limit of large word length, we need to
consider the scaling of the distribution with the word length $n$. The
first step is to determine how the variance scales with $n$. A plot of
the numerical data is given in Figure~\ref{fig:var}, which shows that,
for large $n$, the variance appears to scale linearly with $n$.  A
least squares fit to the data points for $40\le n\le 80$ gives a slope
of $0.021616$.

\begin{figure}[t]
\caption{Distribution of cubefree words of length $80$ as a function
of the excess $e=n_{0}-40$, compared to a Gaussian distribution with
the same variance $\sigma^{2}$.}
\centerline{\includegraphics[width=0.6\textwidth]{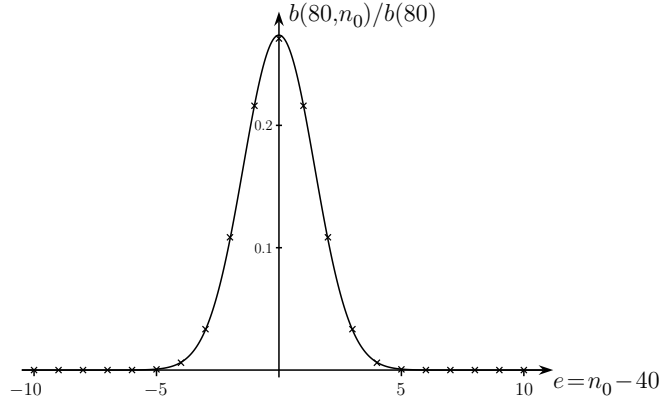}}
\label{fig:dist}
\end{figure}

\begin{figure}[t]
\caption{Variance of the distribution of the letter frequency in
binary cubefree words of length $n$.}
\centerline{\includegraphics[width=0.5\textwidth]{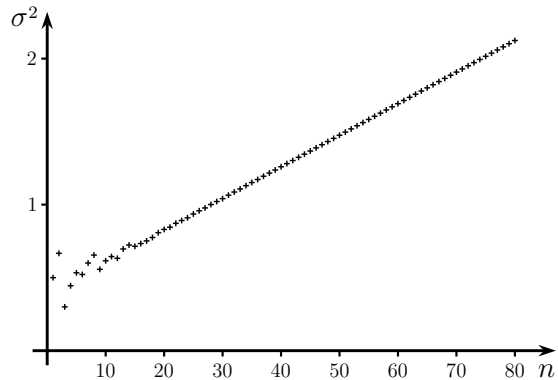}}
\label{fig:var}
\end{figure}

Assuming that the distribution for fixed $n$ is Gaussian, the suitably
re-scaled data
\[
    g_{n}(x) = \sqrt{n} \: \frac{b(n,\tfrac{n}{2}+e)}{b(n)},
\]
considered as a function of the rescaled letter excess
\[
    x = \frac{e}{\sqrt{n}},
\]
should approach a Gaussian distribution
\[
    G(x) = \frac{1}{\sqrt{2\pi\sigma^{2}}}
     \exp\Bigl(-\frac{x^2}{2\sigma^{2}}\Bigr)
\]
with variance $\sigma^{2}\simeq 0.021616$. Figure~\ref{fig:scal} shows
a plot of this distribution, together with the data points obtained
for $40\le n\le 80$. Clearly, there are some deviations, which has to
be expected due to the fact that the relationship between the variance
and the length shown in Figure~\ref{fig:var}, while being
asymptotically linear, is not a proportionality; however, the overall
agreement is reasonable. A plausible conjecture, therefore, is that
the scaled distribution becomes Gaussian in the limit of large word
length. In terms of the entropy, the observed concentration property
is consistent with the entropy maximum occurring at letter frequency
$1/2$, and a lower entropy for other letter frequencies. This is
similar to the observed and conjectured behaviour for ternary
squarefree words in Ref.~\cite{Richard:2004}.

\begin{figure}
\caption{The scaled data $g_{n}(x)$ for lengths $40\le n\le 80$,
compared to a Gaussian distribution $G(x)$.}
\centerline{\includegraphics[width=0.6\textwidth]{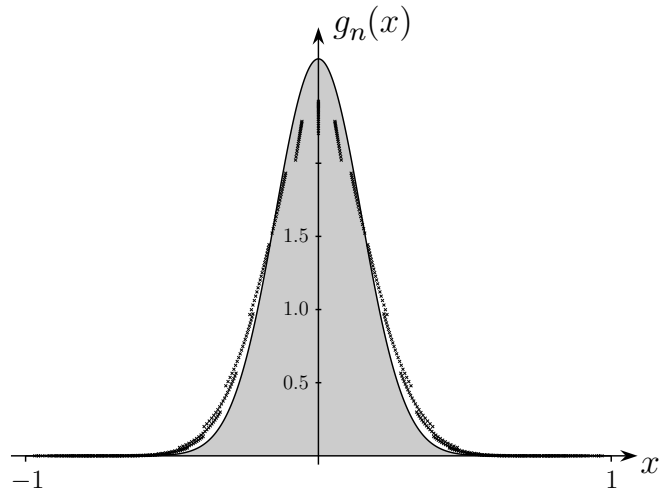}}
\label{fig:scal}
\end{figure}

By an application of Theorem~\ref{thm:gbtf}, the cubefree morphisms of
Eq.~\eqref{eq:mor9} show that the entropy for the case of letter
frequency $f_{0}=f_{1}=1/2$ is positive. More interesting in the case
of non-equal letter frequencies. As an example, consider the
$13$-uniform morphism
\begin{align*}
     a_{11}&\mapsto 0010010110011\\
     a_{12}&\mapsto 0010011010011\\
     a_{21}&\mapsto 0010110010011\\
     a_{22}&\mapsto 0100101001011,
\end{align*}
where all words on the right-hand side comprise seven letters $0$ and
six letters $1$. One can check that this morphism satisfies the
criterion of Theorem~\ref{thm:testset}, hence is cubefree.
Consequently, the matrix $M$ of Theorem~\ref{thm:gbtf} is
$M=\left(\begin{smallmatrix}7&7\\ 6&6 \end{smallmatrix}\right)$, and
the letter frequencies of any word constructed by this morphism are
$f_{0}=7/13$ and $f_{1}=6/13$. Hence, the set of binary cubefree words
with letter frequencies $f_{0}=7/13$ and $f_{1}=6/13$ has positive
entropy bounded by $h\ge \frac{1}{6} \log 2\simeq 0.115525$ (and, by
symmetry, this also holds for $f_{0}=6/13$ and $f_{1}=7/13$). Again,
like in the case of ternary squarefree words discussed in
Ref.~\cite{Richard:2004}, it is plausible to conjecture that the
entropy is positive on an entire interval of letter frequencies around
$1/2$ (where it is maximal), presumably on
$(f_{\text{min}},1-f_{\text{min}})$.

\subsection{Ternary squarefree words}

Letter frequencies in ternary squarefree words were first studied by
Tarannikov \cite{Tarannikov:2002}. He showed that the minimal letter
frequency $f_{\text{min}}$ is bounded by
\[
   0.274649 \simeq \frac{1780}{6481} \le f_{\text{min}}
   \le \frac{64}{233} \simeq 0.274678,
\]
see \cite[Thm.~4.2]{Tarannikov:2002}. These bounds have recently been
improved by Ochem \cite{Ochem:2007a} to
\[
   0.2746498 \simeq \frac{1000}{3641} \le f_{\text{min}}
   \le \frac{883}{3215} \simeq  0.2746501,
\]
who also showed that the maximum frequency $f_{\text{max}}$ of a
letter in a ternary squarefree word is bounded by
\[
    f_{\text{max}} \le \frac{255}{653} \simeq 0.390505;
\]
see \cite[Thm.~1]{Ochem:2007a}. Very recently, Khalyavin
\cite{Khalyavin:2007} proved that the minimum frequency is indeed equal
to Ochem's upper bound, so
\[
     f_{\text{min}}  =  \frac{883}{3215},
\]
which finally settles this question.

By constructing suitable squarefree morphisms in accordance with
Theorem~\ref{thm:gbtf}, Richard and Grimm \cite{Richard:2004} showed
that, for a number of letter frequencies, the number of ternary
squarefree words grows exponentially. This has recently been further
investigated by Ochem \cite{Ochem:2007b}.

\section{Summary and Outlook}

In this paper, we reviewed recent progress on the combinatorics of
$k$-powerfree words, with particular emphasis on the examples of
binary cubefree and ternary squarefree words, which have attracted
most attention over the years. Recent work in this area, using
extensive computer searches, but also new methods, has led to a
drastic improvement of the known bounds for the entropy of these
sets. No analytic expression for the entropy is known to date, and the
results on the generating function for the sets of length-$p$
powerfree words indicate that this may be out of reach. However,
considerable progress has been made on other combinatorial questions,
such as letter frequencies, where again bounds have been improved, but
eventually also a definite answer has emerged, in this case on the
minimum letter frequency in ternary squarefree words.

We also presented some new results on binary cubefree words, including
an enumeration of the number of words and their letter frequencies for
length up to $80$. The empirical distribution of the number of words
as a function of the excess of one letter is investigated, and
conjectured to become Gaussian in the limit of infinite word lengths
after suitable scaling. We also found bounds on the letter frequency
in binary squarefree words, and show that exponentially many words
with unequal letter frequency exist, like in the case of ternary
squarefree words.  The analysis of the generating functions of
length-$p$ binary cubefree words, which we calculated for $p\le 14$,
also shows striking similarity to the case of ternary squarefree
words, suggesting that the observed behaviour may be generic for sets
of $k$-powerfree words.

While a lot of progress has been made, there remain many open
questions. For instance, is there an explanation for the observed
accumulation of poles and zeros of the generating functions on or near
the unit circle, and is it possible to prove what happens in the limit
when $p\to\infty$? How does the entropy depend on the power, say for
binary $k$-powerfree words? A partial answer to this question is given
in Ref.~\cite{Karhu:2004}, but it would be nice to show that, at least
in some region, the entropy increases by a finite amount at any
rational value of $k$, which you might expect to happen. Concerning
powerfree words with given letter frequencies, how does the entropy
vary as a function of the frequency? One might conjecture that the
entropy changes continuously, but at present all we have are results
that for some very specific frequencies, where powerfree morphisms have
been found, the entropy is positive. Some of these questions may be too
hard to hope for an answer in full generality, but the recent progress
in the area shows that one should keep looking for alternative
approaches which may succeed.

\section*{Acknowledgements}

This research has been supported by EPSRC grant EP/D058465/1. The
authors thank Christoph Richard for useful comments and discussions.
UG would like to thank the University of Tasmania (Hobart) and the
AMSI (Melbourne) for their kind hospitality, and for the opportunity to
present his research as part of the AMSI/MASCOS Theme Program
\textit{Concepts of Entropy and Their Applications}.

\bibliographystyle{unsrt}

\begin{thebibliography}{99}\itemsep 2pt

\bibitem{Thue:1977}
Thue, A.
\newblock {\em Selected mathematical papers};
\newblock Nagell, T., Selberg, A., Selberg, S., Thalberg, K., Eds.;
\newblock Universitetsforlaget: Oslo, 1977.

\bibitem{Morse:1921}
Morse, H.~M.
\newblock Recurrent geodesics on a surface of negative curvature.
\newblock {\em Trans. Amer. Math. Soc.} 22 (1921)~84--100.

\bibitem{Lothaire:1997}
Lothaire, M.
\newblock {\em Combinatorics on words}.
\newblock Cambridge Mathematical Library. Cambridge University Press:
  Cambridge, 1997;
\newblock corrected reprint.

\bibitem{Lothaire:2002}
Lothaire, M.
\newblock {\em Algebraic combinatorics on words}.
\newblock Encyclopedia of Mathematics and its Applications. Cambridge
  University Press: Cambridge, 2002.

\bibitem{Lothaire:2005}
Lothaire, M.
\newblock {\em Applied combinatorics on words}.
\newblock Encyclopedia of Mathematics and its Applications. Cambridge
  University Press: Cambridge, 2005.

\bibitem{Zech:1958}
Zech, T.
\newblock Wiederholungsfreie Folgen.
\newblock {\em Z. angew. Math. Mech.} 38 (1958) 206--209.

\bibitem{Pleasants:1970}
Pleasants, P. A.~B.
\newblock Nonrepetitive sequences.
\newblock {\em Proc. Cambr. Philos. Soc.} 68 (1970) 267--274.

\bibitem{Bean:1979}
Bean, D.; Ehrenfeucht, A.; McNulty, G.
\newblock Avoidable patterns in strings of symbols.
\newblock {\em Pacific J. Math.} 95 (1979) 261--294.

\bibitem{Crochemore:1982}
Crochemore, M.
\newblock Sharp characterizations of squarefree morphisms.
\newblock {\em Theoret. Comput. Sci.} 18 (1982) 221--226.

\bibitem{Shelton:1983}
Shelton, R.~O.
\newblock On the structure and extendibility of squarefree words.
\newblock In {\em Combinatorics on Words};
Cummings, J.~L., Ed.; Academic Press: Toronto, 1983;
pp.~101--118. 

\bibitem{Brandenburg:1983}
Brandenburg, F.-J.
\newblock Uniformly growing k-th power-free homomorphisms.
\newblock {\em Theoret. Comput. Sci.} 23 (1983) 69--82.

\bibitem{Brinkhuis:1983}
Brinkhuis, J.
\newblock Nonrepetitive sequences on three symbols.
\newblock {\em Quart. J. Math. Oxford Ser. (2)} 34 (1983) 145--149.

\bibitem{Leconte:1985a}
Leconte, M.
\newblock A characterization of power-free morphisms.
\newblock {\em Theoret. Comput. Sci.} 38 (1985) 117--122.

\bibitem{Leconte:1985b}
Leconte, M.
\newblock $k$-th power-free codes.
\newblock In {\em Automata on Infinite Words}; Nivat, M., Perrin, D., Eds.; 
Springer: Berlin, 1985;
  Lecture Notes in Computer Science, Vol.~192, pp.~172--187.

\bibitem{Seebold:1985}
S\'{e}\'{e}bold, P.
\newblock Overlap-free sequences.
\newblock In {\em Automata on Infinite Words}; Nivat, M., Perrin, D., Eds.; 
Springer: Berlin, 1985;
  Lecture Notes in Computer Science, Vol.~192, pp.~207--215.

\bibitem{Kobayashi:1986}
Kobayashi, Y.
\newblock Repetition-free words.
\newblock {\em Theoret. Comput. Sci.} 44 (1986) 175--197.

\bibitem{Keraenen:1987}
Ker{\"a}enen, V.
\newblock On the k-freeness of morphisms on free monoids.
\newblock {\em Lecture Notes in Computer Science} Vol.~247, 1987, pp.~180--188.

\bibitem{Baker:1989}
Baker, K.; McNulty, G.; Taylor, W.
\newblock Growth problems for avoidable words.
\newblock {\em Theoret. Comput. Sci.} 69 (1989)~319--345.

\bibitem{Currie:1993}
Currie, J.
\newblock Open problems in pattern avoidance.
\newblock {\em Amer. Math. Monthly} 100 (1993) 790--793.

\bibitem{Kolpakov:1999}
Kolpakov, R.; Kucherov, G.; Tarannikov, Y.
\newblock On repetition-free binary words of minimal density.
\newblock {\em Theoret. Comput. Sci.} 218 (1999) 161--175.

\bibitem{Grimm:2001}
Grimm, U.
\newblock Improved bounds on the number of ternary square-free words.
\newblock {\em J. Integer Seq.} 4 (2001), Article 01.2.7.

\bibitem{Currie:2002}
Currie, J.
\newblock There are circular square-free words of length $n$ for $n\ge 18$.
\newblock {\em Electron. J. Combin.} 9 (2002) \#N10.

\bibitem{Richomme:2002}
Richomme, G.; Wlazinski, F.
\newblock Some results on k-power-free morphisms.
\newblock {\em Theoret. Comput. Sci.} 273 (2002) 119--142.

\bibitem{Kucherov:2003}
Kucherov, G.; Ochem, P.; Rao, M.
\newblock How many square occurrences must a binary sequence contain?
\newblock {\em Electron. J. Combin.} 10 (2003) \#R12.

\bibitem{Karhu:2004}
{Karhum\"{a}ki}, J.; Shallit, J.
\newblock Polynomial versus exponential growth in repetition-free binary words.
\newblock {\em J. Combin. Theory Ser. A} 105 (2004) 335--347.

\bibitem{Richard:2004}
Richard, C.; Grimm, U.
\newblock On the entropy and letter frequencies of ternary square-free words.
\newblock {\em Electron. J. Combin.} 11 (2004) \#R14.

\bibitem{Ochem:2006}
Ochem, P.; Reix, T.
\newblock Upper bound on the number of ternary square-free words.
\newblock {\em Workshop on Words and Automata (WOWA'06), St.\ Petersburg} 2006.

\bibitem{Richomme:2007}
Richomme, G.; Wlazinski, F.
\newblock Existence of finite test-sets for $k$-powerfreeness of uniform
  morphims.
\newblock {\em Discrete Applied Math.} 155 (2007) 2001--2016.

\bibitem{Kolpakov:2007}
Kolpakov, R.
\newblock Efficient lower bounds on the number of repetition-free words.
\newblock {\em J. Integer Seq.} 19 (2007), Article 07.3.2.

\bibitem{Ochem:2007a}
Ochem, P.
\newblock Letter frequency in infinite repetition-free words.
\newblock {\em Theoret. Comput. Sci.} 380 (2007) 388--392.

\bibitem{Chalopin:2007}
Chalopin, J.; Ochem, P.
\newblock Dejean's conjecture and letter frequency.
\newblock {\em Electronic Notes in Discrete Mathematics} 2007, pp.~501--505.

\bibitem{Ochem:2007b}
Ochem, P.
\newblock Unequal letter frequencies in ternary square-free words.
\newblock {\em Proceedings of 6th International Conference on Words (WORDS
  2007), Marseille} 2007.

\bibitem{Khalyavin:2007}
Khalyavin, A.
\newblock The minimal density of a letter in an infinite ternary square-free
  word is {$\frac{883}{3215}$}.
\newblock {\em J. Integer Seq.} 10 (2007), Article 07.6.5.

\bibitem{Queffelec:1987}
Queff{\'e}lec, M.
\newblock {\em Substitution dynamical systems -- spectral analysis};
\newblock Springer-Verlag: Berlin, 1987; 
\newblock Lecture Notes in Mathematics.

\bibitem{Fogg:2002}
Fogg, N.~P.
\newblock {\em Substitutions in dynamics, arithmetics and combinatorics};
\newblock V.\ Berth\'e, S.\ Ferenczi, C.\ Mauduit and A.\ Siegel, Eds.; 
\newblock Springer-Verlag: Berlin, 2002;
\newblock Lecture Notes in Mathematics. 

\bibitem{Allouche:2003}
Allouche, J.-P.; Shallit, J.
\newblock {\em Automatic sequences};
\newblock Cambridge University Press: Cambridge, 2003.

\bibitem{Moody:1997}
Moody, R.~V. 
\newblock {\em The mathematics of long-range aperiodic order}; 
\newblock Kluwer Academic Publishers Group: Dordrecht, 1997;
\newblock NATO Advanced Science Institutes Series C: Mathematical and Physical
  Sciences, Vol.~489.

\bibitem{Keraenen:1985}
Ker{\"a}nen, V.
\newblock On k-repetition free words generated by length uniform morphisms over
  a binary alphabet.
\newblock {\em Lecture Notes in Computer Science} 194 (1985) 338--347.

\bibitem{Baake:1997}
Baake, M.; Elser, V.; Grimm, U.
\newblock The entropy of square-free words.
\newblock {\em Math. Comput. Modelling} 26 (1997) 13--26.

\bibitem{Walters}
Walters, P.
\newblock {\em An introduction to ergodic theory};
\newblock Springer-Verlag: New York, 1982.

\bibitem{Noonan:1999}
Noonan, J.; Zeilberger, D.
\newblock The goulden-jackson cluster method: Extensions, applications, and
  implementations.
\newblock {\em J. Difference Eq. Appl.} 5 (1999) 355--377.

\bibitem{Berstel:2005}
Berstel, J.
\newblock Growth of repetition-free words -- a review.
\newblock {\em Theoret. Comput. Sci.} 340 (2005) 280--290.

\bibitem{Elser:1983}
Elser, V.
\newblock Repeat-free sequences.
\newblock {\em Lawrence Berkeley Laboratory report} 1983, {\em
  LBL-16632}.

\bibitem{Sun:2003}
Sun, X.
\newblock New lower-bound on the number of ternary square-free words.
\newblock {\em J Integer Seq.} 6 (2003), Article 03.3.2.

\bibitem{Edlin:1999}
Edlin, A.~E.
\newblock The number of binary cube-free words of length up to {$47$} and their
  numerical analysis.
\newblock {\em J. Differ. Equations Appl.} 5 (1999) 353--354.

\bibitem{Titchmarsh:1976}
Titchmarsh, E.~C.
\newblock {\em The theory of functions};
\newblock Oxford University Press: Oxford, 1976.

\bibitem{Ekhad:1998}
Ekhad, S.; Zeilberger, D.
\newblock There are more than $2^{n/17}$ $n$-letter ternary square-free words.
\newblock {\em J. Integer Seq.} 1 (1998), Article 98.1.9.

\bibitem{BGJ}
Baake, M.; Grimm, U.; Joseph, D.
\newblock Trace maps, invariants, and some of their applications.
\newblock {\em Int. J. Mod. Phys. B} 7 (1993) 1527--1550.

\bibitem{Tarannikov:2002}
Tarannikov, Y.
\newblock The minimal density of a letter in an infinite ternary square-free
  words is $0.2746\ldots$.
\newblock {\em J. Integer Seq.} 5 (2002), Article 02.2.2.

\end{thebibliography}

\end{document}